\documentclass[11pt, a4paper]{amsart}

\usepackage{amsfonts,amsmath,amssymb, amscd,fullpage}
\usepackage{stmaryrd}
\usepackage[all]{xy}

\newtheorem{theorem}{Theorem}[section]
\newtheorem{lemma}[theorem]{Lemma}
\newtheorem{definition}[theorem]{Definition}
\newtheorem{proposition}[theorem]{Proposition}
\newtheorem{corollary}[theorem]{Corollary}
\newtheorem{example}[theorem]{Example}
\newtheorem{remark}[theorem]{Remark}

\theoremstyle{definition}

\newcommand\pf{\begin{proof}}
\newcommand\epf{\end{proof}}

\DeclareMathOperator{\id}{id}

\numberwithin{equation}{section}

\title{Quotients and Hopf images of a smash coproduct}

\author{Julien Bichon}
\address{Laboratoire de Math\'ematiques,
Universit\'e Blaise Pascal,
Complexe universitaire des C\'ezeaux,
3 place Vasarely, 
63178~Aubi\`ere Cedex, France}
\email{Julien.Bichon@math.univ-bpclermont.fr}

\subjclass[2010]{16T05}

\begin{document}

\begin{abstract}
We describe the Hopf algebra quotients and Hopf images of the smash coproduct of a group algebra by the algebra of functions on a finite group. 
\end{abstract}

\maketitle

\section{Introduction}

The smash coproduct, associated to an action of a finite group on a discrete group, is one of the most well-known constructions to produce non-commutative and non-cocommutative Hopf algebras.
The aim of this paper is to provide a description of the Hopf algebra quotients of such a smash coproduct.

Let us first recall  the construction.
Let $H\curvearrowright\Gamma$ be a finite group $H$ acting by automorphisms on a discrete group $\Gamma$. Then
the smash  coproduct Hopf algebra $k[\Gamma] \rtimes k^H$ ($k$ denotes an arbitrary field) is 
$k[\Gamma]\otimes k^H$ as an algebra, where $k[\Gamma]$ denotes the (convolution) group algebra of $\Gamma$ and $k^H$ is the algebra of $k$-valued functions of $H$, and the  comultiplication is  given by 
$$\Delta(r\#\delta_h)=\sum_{l\in H}r\#\delta_{l}\otimes l^{-1}.r  \#\delta_{l^{-1}h}
= \sum_{l\in H}r\#\delta_{l^{-1}}\otimes l.r  \#\delta_{lh}$$ for $r\in\Gamma,h \in H$ (we denote by $r\#\delta_h$ the element $r \otimes \delta_h$ of  $k[\Gamma] \rtimes k^H$).
The Hopf algebra  $k[\Gamma] \rtimes k^H$  fits into an exact sequence of Hopf algebras (see \cite{ade})
$$k \rightarrow k^H  \rightarrow k[\Gamma] \rtimes k^H \rightarrow k[\Gamma] \rightarrow k$$ 
Now if $L$ is Hopf algebra quotient of  $k[\Gamma] \rtimes k^H$, some standard arguments show that $L$ fits into an exact sequence 
 $$k \rightarrow k^{G}  \rightarrow L \rightarrow k[\overline{\Gamma}] \rightarrow k$$ 
where $G\subset H$ is a subgroup  and $\overline{\Gamma}$ is a quotient of $\Gamma$. Moreover, this exact sequence is cleft, so the general theory of cleft extensions (see \cite{ade,and96,mas}) ensures that $L$ is isomorphic to a general bismash product $k^G{^{\tau}\!}\#_{\sigma}k[\overline{\Gamma}]$, involving complicated cohomological data, that are known to be difficult to deal with in general (see \cite{mas95} for an illustration of a situation where it is better to forget about the whole structure of the bismash product).   

Instead of a bismash product, we propose to use the notion of  quotient datum to describe the quotients of $k[\Gamma] \rtimes k^H$: a quotient datum  is a triple
 $(G,N,\Phi)$ where $G$ is a subgroup of $H$,
$N \triangleleft \Gamma$ is a normal and $G$-stable subgroup of $\Gamma$, and $\Phi : N \rightarrow (k^G)^{\times}$ is a group morphism satisfying some simple conditions.
To a quotient datum $(G,N,\Phi)$ we associate a Hopf algebra $k[\Gamma/N] \rtimes_\Phi k^G$, which is a quotient of $k[\Gamma] \rtimes k^H$, and show conversely that any Hopf algebra quotient of $k[\Gamma] \rtimes k^H$ is isomorphic to $k[\Gamma/N] \rtimes_\Phi k^G$ for some quotient datum $(G,N,\Phi)$.

It seems that the notion is simple enough to allow concrete description of the quotients of $k[\Gamma] \rtimes k^H$, at least of course when the normal subgroup structure of $\Gamma$ is not too complicated, and we examine some examples to illustrate this.

The original motivation for this work came from the following problem.

 First recall \cite{bb2}
that for a Hopf algebra representation $\pi : A \rightarrow {\rm End}(V)$ on a vector space $V$, there exists a unique Hopf algebra $L$, called the Hopf image of $\pi$, that produces a minimal factorization 
$$\xymatrix{A\ar[rr]^{\pi}\ar[rd]&&{\rm End}(V)\\&L\ar[ur]&}$$
When $A=k[\Gamma]$ is a group algebra, then $L=k[\Gamma/{\rm Ker}(\pi)]$, and hence the problem of computing the Hopf image amounts to computing the kernel of the group representation, which of course can be quite difficult.
Techniques for computing Hopf images for several classes of Hopf algebras were developed in \cite{bb2}.

Now recall \cite{bni,ban} that to a complex  Hadamard matrix  $H \in M_N(\mathbb C)$ is associated a representation $\pi_H : A_s(N) \rightarrow M_N(\mathbb C)$ of Wang's quantum permutation algebra $A_s(N)$ \cite{wan} (the universal cosemisimple Hopf algebra coacting on the diagonal algebra $k^N$ when $k$ has characteristic zero \cite{bic}), whose Hopf image is thought of as representing the quantum symmetry group of the Hadamard matrix or of the corresponding subfactor (see \cite{jsu}). It is in general very difficult to compute the Hopf image of $\pi_H$. The case $H=F_M\otimes_QF_N$ of the tensor product of Fourier matrices deformed by a matrix of coefficients $Q$ (\cite{dit}) was studied in \cite{bb3}, and a factorization of $\pi_H$ through a certain smash coproduct 
$\mathbb C[\Gamma] \rtimes \mathbb C^{{\mathbb Z}_M}$ was found there, which was shown to be the Hopf image under a genericity assumption on $Q$. However the general case remained unclear, and after analysing the situation, it became clear that it was in fact not more difficult to try to describe all the possible quotients of the crossed coproduct and only after that, try to identify the Hopf image. From these considerations we get a method to compute the Hopf image of a smash coproduct in general, described in Section 4, that enables us to make more precise some of the results of \cite{bb3} in special situations. In particular we show that if $M=2$ and $N$ is prime, or $N=2$ and $M$ is prime, the genericity assumption in \cite{bb3} can be weakened to the assumption that one of the coefficients of the parameter matrix $Q$ is not a root of unity. 

The paper is organized as follows. In Section \ref{sec:quotdata}  we define quotient data and  describe the Hopf algebra quotients of the smash coproduct  of a group algebra by the algebra of functions on a finite group in terms of Hopf algebras associated to quotient data. In Section \ref{sec:example} we discuss some examples. In Section \ref{sec:hopfim}, after having recalled the basic notions around Hopf images, we provide a general method, based on the previous considerations, to compute Hopf images for smash coproducts. The final Section \ref{sec:exhopim} is devoted to examples of computations of Hopf images, providing in particular cases  refinements of some results of \cite{bb3}.

\smallskip

\textbf{Notations and conventions.}
We work over an arbitrary field $k$. 
We assume that the reader is familiar with the basic theory of Hopf algebras, see  \cite{mon} for example.
If $A$ is a Hopf algebra, as usual, $\Delta$, $\varepsilon$ and $S$ stand respectively for the comultiplication, counit and antipode of $A$. If $\Gamma$ is a group, we denote by $k[\Gamma]$ the (convolution) group algebra having its group-like elements identified with the elements of $\Gamma$, and if $H$ is a finite group, we denote by $k^H$ the Hopf algebra of functions on $H$, i.e. $k^H=k[H]^*$ as Hopf algebras, see e.g. Chapter 1 in \cite{mon}. 

\smallskip

\textbf{Acknowledgements.} This paper is a continuation of a long collaboration with Teodor Banica on the topics of Section 5. I would like to thank him for many interesting discussions.

\section{Quotient data}\label{sec:quotdata}

Let $H\curvearrowright\Gamma$ be a finite group $H$ acting by automorphisms on a discrete group $\Gamma$. Recall that
the smash  coproduct Hopf algebra is $k[\Gamma] \rtimes k^H=k[\Gamma]\otimes k^H$ as an algebra, with comultiplication, counit and antipode given by 
$$\Delta(r\#\delta_h)=\sum_{l\in H}r\#\delta_{l}\otimes l^{-1}.r  \#\delta_{l^{-1}h}
= \sum_{l\in H}r\#\delta_{l^{-1}}\otimes l.r  \#\delta_{lh}$$
$$\varepsilon(r\#\delta_h) = \delta_{h,1}, \quad S(r\#\delta_h)= h^{-1}.r^{-1}\#
\delta_{h^{-1}}$$
 for $r\in\Gamma,h \in H$.

The precise definition of a quotient datum for $H\curvearrowright\Gamma$ is as follows.

\begin{definition}
 Let $H\curvearrowright\Gamma$ as above. A quotient datum for $H\curvearrowright\Gamma$ is a triple $(G,N,\Phi)$ where
\begin{enumerate}
\item $G \subset H$ is a subgroup.
 \item $N \triangleleft \Gamma$ is a normal and $G$-stable subgroup of $\Gamma$.
\item $\Phi : N \rightarrow (k^G)^{\times}$ is a group morphism such that $$\Phi(r)(lh)=\Phi(l^{-1}.r)(h)\Phi(r)(l), \quad
\Phi(r)= \Phi(srs^{-1})$$   for any $ r \in N$, $s \in \Gamma$, $h,l \in G$.
\end{enumerate}
We denote by ${\rm QD}(H\curvearrowright\Gamma)$ the set of quotient data for $H\curvearrowright\Gamma$.
\end{definition}

\begin{example}\label{ex:char}
If $G\subset H$ is a subgroup, $N \triangleleft \Gamma$ is a normal and $G$-stable subgroup of $\Gamma$ and $\Phi : N \rightarrow \widehat{G}={\rm Hom}(G,k^\times)$
is a group morphism such that $\Phi(r)= \Phi(srs^{-1})$ and $\Phi(h.r)=\Phi(r)$ for any $ r \in N$, $s \in \Gamma$,
$h \in G$, then $(G,N,\Phi) \in {\rm QD}(H\curvearrowright\Gamma)$.
\end{example}

See the end of the next section for an example of a quotient datum that is not of the type of the previous example.

The proof of the following easy lemma, that we record for future use, is left to the reader.

\begin{lemma}\label{qdeasy}
 Let $H\curvearrowright\Gamma$ as above and let $(G,N,\Phi)\in {\rm QD}(H\curvearrowright\Gamma)$.
\begin{enumerate}
 \item  For $r, s \in \Gamma$ with $rs \in N$, we have $sr \in N$ and $\Phi(rs)=\Phi(sr)$.
\item For $h \in G$ and $r \in N$,
 we have $\Phi(r)(1)=1$ and $\Phi(h. r)(h)= \Phi(r^{-1})(h^{-1})$.
\end{enumerate}
\end{lemma}

We now associate a quotient Hopf algebra of $k[\Gamma] \rtimes k^H$ to a quotient datum for $H \curvearrowright \Gamma$.

\begin{proposition}\label{construction}
  Let $H\curvearrowright\Gamma$ as above and let $(G,N,\Phi)\in {\rm QD}(H\curvearrowright\Gamma)$.
Choose a section $j : \Gamma/N \rightarrow \Gamma$ of the canonical projection $u: \Gamma \rightarrow \Gamma/N$, with $ju(1)=1$.
The following formulas 
$$(u(r)\#\delta_h)(u(s)\#\delta_k)= u(rs)\#\delta_h\delta_k \Phi(ju(r)ju(s)ju(rs)^{-1})$$
$$\Delta(u(r)\#\delta_h)= \sum_{l\in G}\left(u(r)\#\delta_{l^{-1}}\right) \otimes u(l.r)\#
\delta_{lh}\Phi(l.ju(r) ju(l. r)^{-1})$$
$$S(u(r)\#\delta_h)= u(h^{-1}.r^{-1})\#
\Phi(ju(h^{-1}.r^{-1})^{-1}h^{-1}.ju(r)^{-1})\delta_{h^{-1}}$$
together with the obvious unit and counit define a Hopf algebra structure on $k[\Gamma/N] \otimes k^G$, which, up to isomorphism, does not depend on the choice of $j$.
We denote by $k[\Gamma/N] \rtimes_\Phi k^G$ the resulting Hopf algebra. Moreover the map
\begin{align*}
 q : k[\Gamma] \rtimes  k^H&\longrightarrow k[\Gamma/N] \rtimes_\Phi k^G \\
r\#\delta_h &\longmapsto u(r)\#\Phi(rju(r)^{-1})\delta_{h_{|G}}
\end{align*}
is a surjective Hopf algebra map.
\end{proposition}

\begin{proof}
 The first thing to do is to check that the above multiplication, comultiplication and antipode on  $k[\Gamma/N]
\otimes k^G$ are well-defined: this is easily done.
The map $q$ is clearly surjective, so to check that the multiplication and comultiplication just defined on 
  $k[\Gamma/N]
\otimes k^G$ are indeed associative and co-associative, it is enough to check that they are preserved by $q$. For $h,k \in H$ and $r,s \in \Gamma$, we have
\begin{align*}
 q(r\#\delta_h) \cdot q(s\#\delta_k)&= u(r)\#\delta_{h_{|G}}\Phi(rju(r)^{-1}) \cdot u(s)\#\delta_{k_{|G}}\Phi(sju(s)^{-1})\\
&= u(rs)\#(\delta_h\delta_k)_{|G} \Phi(rju(r)^{-1})\Phi(sju(s)^{-1})\Phi(ju(r)ju(s)ju(rs)^{-1})\\
&= u(rs)\#(\delta_h\delta_k)_{|G} \Phi(sju(s)^{-1})\Phi(rju(s)ju(rs)^{-1}) \\
 &= u(rs) \#(\delta_h\delta_k)_{|G} \Phi(sju(s)^{-1})\Phi(ju(s)ju(rs)^{-1}r)\\
&= u(rs)\#(\delta_h\delta_k)_{|G} \Phi(sju(rs)^{-1}r) = u(rs)\#(\delta_h\delta_k)_{|G} \Phi(rsju(rs)^{-1})\\
&=q(r\#\delta_{h} \cdot s\#\delta_{k})
\end{align*}
where we have used Lemma \ref{qdeasy}.
Using  Lemma \ref{qdeasy} again , we  have for $r \in \Gamma$ and $h \in G$
\begin{align*}
 (&q\otimes q )\Delta(r\#\delta_h) = q \otimes q\left(\sum_{l\in H}r\#\delta_{l^{-1}}\otimes l.r\#\delta_{lh}\right) \\
&= \sum_{l\in G} u(r)\#\delta_{l^{-1}}\Phi( r ju(r)^{-1})\otimes u(l.r)\#\delta_{lh}\Phi(l.rju(l.r)^{-1}) \\
&= \sum_{l\in G} \Phi(rju(r)^{-1})(l^{-1})\Phi(l.rju(l.r)^{-1})(lh)u(r)\#
\delta_{l^{-1}}\otimes u(l.r)\#\delta_{lh} \\
&= \sum_{l\in G} \Phi(l.ju(r)l.r^{-1})(l) \Phi(l.rju(l.r)^{-1})(lh)
 u(r)\#\delta_{l^{-1}}\otimes u(l.r)\#\delta_{lh}\\
&= \sum_{l\in G} \Phi(l.ju(r)l.r^{-1})(l)
\Phi(rl^{-1}.ju(l.r)^{-1})(h)\Phi(l.rju(l.r)^{-1})(l)
u(r)\#\delta_{l^{-1}}\otimes u(l.r)\#\delta_{lh} \\
&= \sum_{l\in G} \Phi(l.ju(r)ju(l.r)^{-1})(l)
\Phi(rl^{-1}.ju(l.r)^{-1})(h)
u(r)\#\delta_{l^{-1}}\otimes u(l.r)\#\delta_{lh} \\
\end{align*}
while for $h \not \in G$, we have $(q\otimes q )\Delta(r\#\delta_h)=0$.
On the other hand, denoting again $\Delta$ the new coproduct, we have $\Delta  q(r\#\delta_h)=0$ if $h \not \in G$, and if $h \in G$
\begin{align*}
 &\Delta  q(r\#\delta_h)  = \Delta(u(r)\#\delta_h\Phi(rju(r)^{-1})) = \Phi(rju(r)^{-1})(h)\Delta(u(r)\#\delta_h) \\
&= \Phi(rju(r)^{-1})(h) \sum_{l\in G}u(r)\#\delta_{l^{-1}} \otimes u(l.r) \# \Phi(l.ju(r) ju(l. r)^{-1}) \delta_{lh} \\
& = \sum_{l\in G}\Phi(rju(r)^{-1})(h) \Phi(l. ju(r) ju(l. r)^{-1})(lh)u(r)\#\delta_{l^{-1}}\otimes u(l.r)\#\delta_{lh}\\
& = \sum_{l\in G}\Phi(rju(r)^{-1})(h) \Phi(ju(r) l^{-1}. ju(l. r)^{-1})(h)
\Phi(l. ju(r) ju(l. r)^{-1})(l) u(r)\#\delta_{l^{-1}}\otimes u(l.r)\#\delta_{lh}\\
& =  \sum_{l\in G}\Phi(r l^{-1}. ju(l. r)^{-1})(h)
\Phi(l. ju(r) ju(l. r)^{-1})(l)u(r)\#\delta_{l^{-1}}\otimes u(l.r)\#\delta_{lh}
\end{align*}
Hence $(q\otimes q )\Delta= \Delta q$, and this shows that we indeed get a bialgebra
$k[\Gamma/N] \rtimes_{\Phi, j} k^G$, with the obvious unit and co-unit, and it is easily seen that $S$ defined above is an antipode, so that $k[\Gamma/N] \rtimes_{\Phi, j} k^G$ is a Hopf algebra, and $q$ is a Hopf algebra map. 

Now choose another section $i$ with the same property. It is obvious that the following linear map
\begin{align*}
f:  k[\Gamma/N] \rtimes_{\Phi, j} k^G&\longrightarrow k[\Gamma/N]\rtimes_{\Phi, i} k^G \\ 
u(r)\#\delta_h  & \longmapsto u(r) \#\delta_h\Phi(iu(r)^{-1}ju(r))
\end{align*}
is bijective. Using Lemma \ref{qdeasy}, one checks that
$$u(rs)\#\Phi(ju(r)ju(s)iu(rs)^{-1})\delta_h\delta_k = f\left(u(r)\#\delta_h \cdot u(s)\#\delta_k\right) = f(\delta_h\#u(r))\cdot f(\delta_k\#u(s))$$
\begin{align*}
\Delta f(u(r)\#\delta_h) &= \sum_{l \in G}
\Phi(ju(r)l^{-1}.(iu(l.r)^{-1}))(h)\Phi(l.iu(r)iu(l.r)^{-1})(l) u(r)\#\delta_{l^{-1}} \otimes u(l.r)\delta_{lh} \# u(r) \\
&= (f\otimes f)\Delta(u(r)\#\delta_h)
\end{align*}
Hence $f$ is a Hopf algebra isomorphism.
\end{proof}

We are now going to show that all the quotients of $k[\Gamma] \rtimes k^H$ have the above form. 
Before this, recall that 
a sequence  of Hopf algebra maps
\begin{equation*}k \to B \overset{i}\to A \overset{p}\to L \to
k\end{equation*} is said to be exact \cite{ade} if the following
conditions hold:
\begin{enumerate}\item $i$ is injective, $p$ is surjective and $pi$ = $\varepsilon 1$,
\item $\ker p =Ai(B)^+ =i(B)^+A$, where $i(B)^+=i(B)\cap{\rm Ker}(\varepsilon)$,
\item $i(B) = A^{{\rm co} p} = \{ a \in A:\, (\id \otimes p)\Delta(a) = a \otimes 1
\} = {^{{\rm co} p}A} = \{ a \in A:\, (p \otimes \id)\Delta(a) = 1 \otimes a
\}$. \end{enumerate}

We first need a couple of lemmas.

\begin{lemma}\label{exact}
  Let $H\curvearrowright\Gamma$ as above and let $(G,N,\Phi)\in {\rm QD}(H\curvearrowright\Gamma)$.
Then the Hopf algebra $k[\Gamma/N]\rtimes_\Phi k^G$ fits into an exact sequence of Hopf algebras
$$k \rightarrow k^G \overset{i}\rightarrow  k[\Gamma/N] \rtimes_\Phi k^G \overset{p}\rightarrow k[\Gamma/N] \rightarrow k$$
where $i(f)=1\#f$ and $p = {\rm id} \otimes \varepsilon$. 
\end{lemma}

\begin{proof}
 This is a direct easy verification.
\end{proof}

The next lemma is a generalization of Lemma 4.5 in \cite{bb3}.

\begin{lemma}\label{iso}
 Let $H\curvearrowright\Gamma$ as above and let $(G,N,\Phi)\in {\rm QD}(H\curvearrowright\Gamma)$.
Let $\pi:k[\Gamma/N]\rtimes_\Phi k^G  \rightarrow L$ be a surjective Hopf algebra map, such that $\pi_{|k^G}$ is injective, and such that for $r \in \Gamma$ and $f \in k^G$, we have:
$$\pi(u(r)\#1)=\pi(1 \otimes f) \implies u(r)=1$$ 
where $u : \Gamma \rightarrow \Gamma/N$ is the canonical surjection.
Then $\pi$ is an isomorphism.
\end{lemma}

\begin{proof}
 The proof, that we include for the sake of completeness, is essentially the same as the one of Lemma 4.5 in \cite{bb3}. We start with the previous exact sequence
$$k \rightarrow k^G \overset{i}\rightarrow  k[\Gamma/N] \rtimes_\Phi k^G \overset{p}\rightarrow k[\Gamma/N] \rightarrow k$$
and put $A=  k[\Gamma/N] \rtimes_\Phi k^G$.
Since $\pi i$ is injective and the Hopf subalgebra $\pi i(k^G)$ is central in $L$, we can form the quotient Hopf algebra $\overline{L} = L/ (\pi i(k^H))^+L$, and we get another exact sequence:
$$k\to k^G\xrightarrow{\pi i} L \overset{q}\to \overline{L}  \to k$$ 
This sequence is indeed exact, e.g. by centrality (see \cite{ade, sch}). So we get the following commutative diagram with exact rows, with the Hopf algebra map on the right surjective:
$$\begin{CD}
k @>>>k^G@>{i}>>  A @>{p}>>k[\Gamma/N]@>>>k\\
@.@| @VV{\pi}V @VV{}V\\
k @>>>k^G@>{\pi i}>> L @>{q}>> \overline{L}@>>>k
\end{CD}$$
Since a quotient of a group algebra is still a group algebra, we get a commutative diagram with exact rows as follows:
$$\begin{CD}
k @>>>k^H@>{i}>>  A @>{p}>>k[\Gamma/N]@>>>k\\
@.@|@VV{\pi}V @VV{}V\\
k @>>>k^H@>{\pi i}>> L @>{q'}>>k[\overline{\Gamma/N}]@>>>k
\end{CD}$$

Here the vertical Hopf algebra map on the right is induced by a surjective group morphism $v : \Gamma/N \rightarrow \overline{\Gamma/N}$, $u(r) \mapsto \overline{u(r)}$. By the short five lemma (see e.g. \cite{mas}, or \cite{aga}) we just have to show that $v$ is injective. 

For $r \in \Gamma$, put:
$$_{u(r)}A= \{a \in A \ | \ p(a_{(1)}) \otimes a_{(2)}= u(r) \otimes a\}$$
$$_{\overline{u(r)}}L= \{l \in L \ | \ q'(l_{(1)}) \otimes l_{(2)}= \overline{u(r)} \otimes l\}$$
The commutativity of the right square ensures that $\pi(_{u(r)}A) \subset {_{\overline{u(r)}}L}$.

Now let $r \in \Gamma$ be such that $vu(r)=1$. We  have $q' \pi(u(r) \# 1) = v p(u(r)\# 1)=vu(r)=\overline{u(r)}=1$, hence $\pi(u(r) \otimes 1) \in {_{\overline{1}}L} = \pi i (k^H)$ (exactness of the sequence), so $\pi(u(r) \otimes 1)= \pi(1 \otimes f)$
for some $f \in k^H$. We conclude by our assumption that $u(r)=1$.
\end{proof}

The following result is the last step towards the determination of the quotients of a smash coproduct, and certainly the most useful in concrete situations.

\begin{proposition}\label{quotients}
 Let $\pi :  k [\Gamma]  \rtimes k^H\rightarrow L$ be a surjective Hopf algebra map with $\pi_{|k^H}$ injective. Then 
there exists  $(H,N,\Phi) \in {\rm QD}(H\curvearrowright\Gamma)$ such that $L$ is isomorphic with  
$k[\Gamma/N]\rtimes_\Phi k^H$.
More precisely, the subgroup $N$ is defined by
$$N= \{ r \in \Gamma \ |  \ \exists f \in (k^H)^{\times} \ {\rm with} \ \pi(r \# 1)= \pi(1 \#  f)\}$$
and for $r \in N$, $\Phi(r)$ is the unique $f \in (k^H)^\times$ such that  $\pi(r \# 1)= \pi(1 \#  f)$.
\end{proposition}

\begin{proof}
 It is immediate that $N$ defined above is a normal subgroup of $\Gamma$.
The above map $\Phi : N \rightarrow (k^H)^\times$ is defined thanks to the injectivity assumption on $\pi_{|k^H}$,
and it is immediate that $\Phi$ is a group morphism with $\Phi(r)= \Phi(srs^{-1})$ for $r \in N$, $s \in \Gamma$.
Let us check that $N$ is $H$-stable. So let $r \in N$.
We have 
\begin{align*}\Delta \pi(r \#1)& =  \sum_{l\in H}\pi(r\#\delta_l)\otimes \pi(l^{-1}.r\#1) 
= \sum_{l\in H} \pi(1\#\delta_{l})\pi(r\#1) \otimes \pi(l^{-1}. r\#1) \\
&=  \sum_{l\in H} \pi(1\#\delta_{l}\Phi(r)) \otimes \pi(l^{-1}. r\#1)=
\sum_{l\in H} \Phi(r)(l) \pi(1\#\delta_{l}) \otimes \pi(l^{-1}. r\#1)
\end{align*}
On the other hand we have
\begin{align*}
 \Delta \pi(1\#\Phi(r))= \Delta \pi (\sum_{h \in H} \Phi(r)(h) 1\#\delta_h)
= \sum_{h,l \in H}  \Phi(r)(h) \pi(1\#\delta_{l}) \otimes \pi(1\#\delta_{l^{-1}h})
\end{align*}
It then follows from the injectivity of $\pi_{|k^H}$ that for any $l \in H$ we have
\begin{align*} \Phi(r)(l)\pi(l^{-1}. r\#1)&=\sum_{h  \in H}  \Phi(r)(h) \pi(1\#\delta_{l^{-1}h})=
\sum_{h  \in H}  \Phi(r)(lh) \pi(1\#\delta_h) \\
&= \pi\left( 1\#\sum_{h  \in H}  \Phi(r)(lh)\delta_h\right)
\end{align*}
It follows that $l^{-1}. r \in N$ and that $\Phi(r)(lh)=\Phi(l^{-1}\cdot r)(h)\Phi(r)(l)$ for any $h \in H$.
Therefore  $(H,N,\Phi) \in {\rm QD}(H\curvearrowright\Gamma)$.

 Let us choose a a section $j : \Gamma/N \rightarrow \Gamma$ of the canonical projection $u : \Gamma \rightarrow \Gamma/N$ with $ju(1)=1$,
and form the Hopf algebra $k[\Gamma/N]\rtimes_\Phi k^H$ as in Proposition \ref{construction}.
Let $q : k[\Gamma]\rtimes k^H \rightarrow k[\Gamma/N]\rtimes_\Phi k^H$ be as in Proposition \ref{construction}, and let $\tilde{\pi} : k[\Gamma/N] \rtimes_\Phi k^H \rightarrow L$ be defined by
$\tilde{\pi}( u(r)\#\delta_h)=\pi( ju(r)\#\delta_h)$.
We have
\begin{align*} 
\tilde{\pi}q(r\# \delta_h)&= \tilde{\pi}(u(r)\# \delta_h\Phi(rju(r)^{-1}))=\pi(ju(r)\#\delta_h\Phi(rju(r)^{-1}))\\
&= \pi(1\#\Phi(rju(r)^{-1}))\pi(ju(r)\#\delta_h)= \pi(rju(r)^{-1})\#1)\pi(ju(r)\#\delta_h)\\
&= \pi(r\#\delta_h)
\end{align*}
  and hence $\tilde{\pi}q=\pi$.  Since $\pi$ and $q$ are surjective Hopf algebra maps, this proves that
$\tilde{\pi}$ is a surjective Hopf algebra map.   We wish to use the previous lemma.
It is clear that $\tilde{\pi}_{|k^H}$ is injective since
 $\pi_{|k^H}$ is. Let $r \in \Gamma$ be such that
$\tilde{\pi}(u(r)\#1)=\tilde{\pi}(1\#f)$ for $f \in k^H$.
Then we have $\pi(ju(r)\#1)=\pi(1\#f)$, and necessarily $f \in (k^H)^\times$ (otherwise there would exist $f'\not=0$ with $f'f=0$ and then $0=\pi(1\#f'f)=\pi(1\#f')\pi(r\#1)$, which would give $\pi(1\#f')=0$ since $\pi(u(r)\#1)$ is invertible).
Hence we have $ju(r)\in N$ and $1=uju(r)=u(r)$: we conclude by the lemma that $\tilde{\pi}$ is injective.
\end{proof}

We arrive at the general description of Hopf algebra quotients of smash coproduct.

\begin{theorem}\label{theo:quot}
 Let $H\curvearrowright\Gamma$ be a finite group $H$ acting by automorphisms on a discrete group $\Gamma$, and let $L$ be a Hopf algebra quotient of the smash coproduct $k[\Gamma] \rtimes k^H$. Then there exists a quotient datum 
 $(G,N,\Phi) \in {\rm QD}(H\curvearrowright\Gamma)$ such that $L$ is isomorphic to $k[\Gamma/N]\rtimes_\Phi k^G$.
\end{theorem}

\begin{proof}
 Let $\pi : k[\Gamma] \rtimes k^H \rightarrow L$ be a surjective Hopf algebra map. Then $\pi(k^H)$ is a Hopf algebra quotient of $k^H$, and hence there exists a subgroup $G \subset H$ such that $\pi$ induces an isomorphism $\pi(k^H) \simeq k^G$. Then there exists a factorization 
$$\xymatrix{ k[\Gamma] \rtimes k^H \ar[rr]^{\pi}\ar[rd]&&L\\&k[\Gamma] \rtimes k^G\ar[ur]_{\pi'}&}$$
where $\pi'_{|k^G}$ is injective, and we conclude by the previous proposition.
\end{proof}

\section{Examples}\label{sec:example}

In order to illustrate the results of the previous section, we now examine a series of examples.

\subsection{First example} We assume in this subsection that ${\rm char}(k)\not=2$.
Let $$\Gamma = D_{\infty}=\mathbb Z_2 * \mathbb Z_2=\langle g_0, g_1 \ | \ g_0^2=1=g_1^2\rangle$$ with the $\mathbb Z_2=\langle h \rangle$-action defined by $h.g_0=g_1$ and $h.g_1=g_0$. 
The Hopf $*$-algebra quotients of $\mathbb C[\mathbb Z_2 * \mathbb Z_2]\rtimes  \mathbb C^{\mathbb Z_2}$ 
have been determined in \cite{bb1}, where this Hopf algebra is denoted $A_h(2)$.  
The methods of the previous paragraph enables us to get without too much effort the description of all the Hopf algebra quotients, over any field of characteristic $\not= 2$.

For $m\geq 1$, let $N_m=\langle (g_0g_1)^m \rangle\simeq \mathbb Z$: this a normal and $H$-stable subgroup
of $\mathbb Z_2 * \mathbb Z_2$.
We get a family of quotients of $k[\mathbb Z_2 * \mathbb Z_2]\rtimes  k^{\mathbb Z_2}$:
$$A(m)=  k[(\mathbb Z_2 * \mathbb Z_2)/N_m]\rtimes  k^{\mathbb Z_2}\simeq k[D_m] \rtimes k^{\mathbb Z_2}$$
of dimension $4m$, with $A(1)\simeq k^{\mathbb Z_2\times \mathbb Z_2}$, $A(2) \simeq k^{D_4}$ and $A(m)$ non-commutative and non-cocommutative if $m \geq 3$.

Now let $\Phi_m : N_m=\langle (g_0g_1)^m \rangle\simeq \mathbb Z \rightarrow \widehat{\mathbb Z_2} =\langle \chi \rangle$ be the unique group morphism with $\Phi_m((g_0g_1)^m)=\chi$.
We have $\Phi_m(h.(g_0g_1)^m)=\Phi_m((g_0g_1)^{-m})=\chi^{-1}=\chi$, so $(\mathbb Z_2, N_m, \Phi_m) \in {\rm QD}(\mathbb Z_2 \curvearrowright \mathbb Z_2 * \mathbb Z_2)$. 
We get a family of quotients of $k[\mathbb Z_2 * \mathbb Z_2]\rtimes  k^{\mathbb Z_2}$:
$$B(m)= k[(\mathbb Z_2 * \mathbb Z_2)/N_m]\rtimes_{\Phi_m}  k^{\mathbb Z_2}$$
of dimension $4m$, with $B(1)\simeq k^{\mathbb Z_4}$, and $B(m)$ non-commutative and non-cocommutative if $m \geq 2$.
The Hopf algebras $A(m)$ and $B(m)$ were studied by Masuoka in \cite{mas00}, Nikshych \cite{ni}, Suzuki \cite{su}, Vainerman \cite{va}, and probably others.

\begin{proposition}
 The non trivial Hopf algebra quotients of  $k[\mathbb Z_2 * \mathbb Z_2]\rtimes  k^{\mathbb Z_2}$ are:
\begin{enumerate}
 \item $k[D_m]$, $m \geq 1$, $k[D_\infty]$,
\item $A(m) =  k[(\mathbb Z_2 * \mathbb Z_2)/N_m]\rtimes  k^{\mathbb Z_2} \simeq k[D_m] \rtimes  k^{\mathbb Z_2}$, $m \geq 1$,
\item $B(m)=  k[(\mathbb Z_2 * \mathbb Z_2)/N_m]\rtimes_{\Phi_m}  k^{\mathbb Z_2}$, $m \geq 1$.
\end{enumerate}
\end{proposition}

\begin{proof}
 Let $\pi : k[\mathbb Z_2 * \mathbb Z_2]\rtimes  k^{\mathbb Z_2} \rightarrow L$ be a surjective Hopf algebra map with $\dim(L)>1$. If $\pi_{| k^{\mathbb Z_2}}$ is not injective, then it is trivial, and  $L$ is quotient of $k[\mathbb Z_2*\mathbb Z_2]$, and hence is isomorphic to $k[D_m]$ for some $m \geq 1$ or $m=\infty$.
Now assume that $\pi_{| k^{\mathbb Z_2}}$ is injective.
It is not difficult to check that the non-trivial $\mathbb Z_2$-stable normal subgroups of  $\mathbb Z_2 * \mathbb Z_2$
are precisely the $N_m=\langle (g_0g_1)^m \rangle$, $m \geq 1$. Let $\Phi : N_m \rightarrow (k^{\mathbb Z_2)})^\times$
be a group morphism such that $(\mathbb Z_2, N_m,\Phi) \in  {\rm QD}(\mathbb Z_2\curvearrowright \mathbb Z_2 * \mathbb Z_2)$. Let $\lambda \in k^*$ be such that $\Phi((g_0g_1)^m)=\delta_1+\lambda \delta_h$. We have
$$\Phi((g_0g_1)^m)=\Phi(g_0(g_0g_1)^mg_0)=\Phi((g_1g_0)^{m})=\Phi((g_0g_1)^{-m})=\Phi((g_0g_1)^m)^{-1}$$
Hence $\lambda=\lambda^{- 1}$, and either $\Phi$ si trivial or $\Phi= \Phi_m$ as above. We conclude by Proposition \ref{quotients}.
\end{proof}

A rough version of the previous result is as follows.

\begin{corollary}\label{quotinftoy}
 The only  non-trivial infinite-dimensional quotient of $k[\mathbb Z_2 * \mathbb Z_2]\rtimes  k^{\mathbb Z_2}$
is $k[\mathbb Z_2 * \mathbb Z_2]$.
\end{corollary}

\subsection{} A first generalization of the previous example is given by
 $$\Gamma = \mathbb Z_2^{*n}=\langle g_0, g_1, \ldots , g_{n-1} \ | \ g_0^2=1=g_1^2= \cdots g_{n-1}^2\rangle$$
with the $S_n$-action given by permutation of the generators. The Hopf algebra 
$k[\mathbb Z_2^{*n}] \rtimes k^{S_n}$ is considered in \cite{rw}, where the ``easy'' quotients are described. Using Theorem \ref{theo:quot}, we get that a Hopf algebra quotient of $k[\mathbb Z_2^{*n}] \rtimes k^{S_n}$ is isomorphic to $k[\mathbb Z_2^{*n}/N] \rtimes_{\rm \Phi} k^{G}$ where $(G, N, \Phi)\in {\rm QD}(S_n \curvearrowright \mathbb Z_2^{*n})$. As pointed out in \cite{rw}, there are many normal $S_n$-stable subgroups $N \subset \mathbb Z_2^{*n}$.

\subsection{The main example} \label{subsec:main} We now come to the examples that motivated this study. 
Let $M, N \geq 2$ and consider the group
$$\Gamma_{M,N}=<g_0, \ldots,g_{M-1}\ | \ g_0^N=\ldots=g_{M-1}^N=1,[g_{i_1}\cdots g_{i_N},g_{j_1}\cdots g_{j_N}]=1>$$
endowed with the cyclic action of $\mathbb Z_M=\langle h \rangle$ on the generators. 
If $M=N=2$,  we are in the situation of the first example.

The Hopf algebra $k[\Gamma_{M,N}]\rtimes k^{\mathbb Z_M}$ arose in \cite{bb3} from certain representations of Wang's quantum permutation algebra. The following description of $\Gamma_{M,N}$ is given in \cite{bb3}. 

\begin{lemma}\label{lem:gammaMN}
 We have a group isomorphism 
$$ \Gamma_{M,N} \simeq \mathbb Z^{(M-1)(N-1)}\rtimes \mathbb Z_N$$
More precisely, for  $0 \leq i \leq M-1$, $0 \leq c \leq N-1$, put $a_{ic} = g_0^{c-1} g_i g_0^{-c}$, and let $T$ be the subgroup of $\Gamma_{M,N}$ generated by the elements $a_{ic}$. Then $T$ is a free abelian group of rank $(M-1)(N-1)$, with basis $\{a_{ic}, \ 1 \leq i \leq M-1, \ 1 \leq c \leq N-1 \}$, and 
there is a split exact sequence 
$$1 \rightarrow T \rightarrow \Gamma_{M,N} \rightarrow \mathbb Z_N \rightarrow 1$$
where the group morphism on the right $\Gamma_{M,N} \rightarrow \mathbb Z_N=\langle t\rangle$ is defined by $g_i \mapsto t$. The $\mathbb Z_N=\langle t \rangle$-action on $T$ is given by $t\cdot a_{ic}= g_0a_{ic}g_0^{-1} = a_{i,c+1}$, while the $\mathbb Z_M=\langle h \rangle$-action on $\Gamma_{M,N}$ is given by $h\cdot a_{ic}=a_{i+1,c}a_{1,c}^{-1}$, $h\cdot g_0=g_0a_{10}$.
\end{lemma}

\begin{proof}
 Let $T$ be the kernel of the above group morphism $\Gamma_{M,N} \rightarrow \mathbb Z_N=\langle t\rangle$. It is clear that $T$ is generated by the elements of type $g_{i_1} \cdots g_{i_N}$, and hence is abelian. The elements $a_{ic}$ belong to $T$, and let $T_0$ be the subgroup generated by these elements. Using the relations
$$g_i a_{jc} g_i^{-1} = a_{j,c+1}, \ g_i^{-1} a_{jc} g_i = a_{j,c-1}$$
we see that $T_0$ is normal in $\Gamma_{M,N}$. The elements $a_{i0} = g_0^{-1}g_i$ belong to $T_0$, and hence we have $[\Gamma_{N,M} : T_0] \leq N$. But then
$N = [\Gamma_{N,M} : T] \leq [\Gamma_{N,M} : T_0] \leq N$, and
thus $T_0=T$. That $T$ is generated by $\{a_{ic}, \ 1 \leq i \leq M-1, \ 1 \leq c \leq M-1 \}$ follows from the identities $$a_{0c}=1, \ {\rm for} \ {\rm any} \ c, \ {\rm and} \ \prod_{c=0}^{N-1} a_{ic}=1 \ {\rm for} \ {\rm any} \ i$$ and to prove that $T$ is indeed free one considers a certain representation of $\Gamma_{M,N}$, see \cite{bb3}, or the examples in the last section. The last assertion about the actions is immediate.
\end{proof}



Our main result on the Hopf algebra quotients of $k[\Gamma_{M,N}]\rtimes k^{\mathbb Z_M}$ is the following generalization of Corollary \ref{quotinftoy}.

\begin{theorem}\label{infiquot}
 Let $f : k[\Gamma_{M,N}]\rtimes k^{\mathbb Z_M} \rightarrow A $ be surjective Hopf algebra map with $A$ infinite-dimensional and non-cocommutative. Assume that one of the following conditions holds.
\begin{enumerate}
 \item $N=2$ and $M$ is prime.
\item $M=2$ and $N$ is prime.
\end{enumerate}
Then $f$ is an isomorphism.
\end{theorem}

In other words, the only non-trivial infinite-dimensional quotients of $k[\Gamma_{M,N}]\rtimes k^{\mathbb Z_M}$
are group algebras.

To prove Theorem \ref{infiquot}, we will need a couple of lemmas. 

\begin{lemma}\label{lem:simple}
 Assume that $N$ is a prime number and that $V \subset \mathbb Z^{(M-1)(N-1)} \rtimes \mathbb Z_N$ is a normal  subgroup. If $V \not \subset  \mathbb Z^{(M-1)(N-1)}$, then the quotient group 
 $(\mathbb Z^{(M-1)(N-1)} \rtimes \mathbb Z_N) /V$ is finite and abelian.
\end{lemma}

\begin{proof}
 First note that it is clear from the definition of $\Gamma_{M,N}$ that an abelian quotient is finite, hence we just have to show that $(\mathbb Z^{(M-1)(N-1)} \rtimes \mathbb Z_N) /V$ is abelian. There exists, by the assumption, $a \in \mathbb Z^{(M-1)(N-1)}$ and $1 \leq k \leq N-1$ such that $at^k \in V$. Working in the quotient group, the assumption that $N$ is prime enables us to assume that $k=1$, and hence $at \in V$. Hence the quotient group  $(\mathbb Z^{(M-1)(N-1)} \rtimes \mathbb Z_N)/V$ is generated by the image of the abelian group $\mathbb Z^{(M-1)(N-1)}$, and is abelian.
\end{proof}

\begin{lemma}\label{linearalg}
Let $p$ be a prime number and let $f: \mathbb Q^{p-1} \rightarrow \mathbb Q^{p-1}$ be a $\mathbb Q$-linear map whose matrix in the canonical basis is 
$$ \begin{pmatrix} 0&0&\cdots&0&0&-1\\1 &0 &\cdots &0&0&-1\\ 0&1& \cdots &0& 0 &-1\\ \vdots & \vdots & \vdots & \vdots & \vdots & \vdots\\
0&0&\cdots &1 &0&-1\\0&0&\cdots &0 &1&-1\end{pmatrix} \quad {\rm or} \quad 
\begin{pmatrix} -1&-1&\cdots&-1&-1& -1\\1 &0 &\cdots &0&0&0\\ 0&1& \cdots & 0& 0 &0 \\ \vdots & \vdots & \vdots & \vdots & \vdots &\vdots \\
0&0&\cdots & 1  &0&0\\0 & 0&0&\cdots&1&0\end{pmatrix}$$
Then for any non-zero $u \in \mathbb Q^{p-1}$, the elements $u, f(u), \ldots , f^{p-2}(u) \in \mathbb Q^{p-1}$ are $\mathbb Q$-linearly independent. 
\end{lemma}

\begin{proof}
As usual we view $\mathbb Q^{p-1}$ as a $\mathbb Q[X]$-module by letting $X.v=f(v)$, for any $v \in \mathbb Q^{p-1}$
 The first matrix is the companion matrix of the cyclotomic polynomial $$P(X)=1+ X + \cdots +X^{p-2} + X^{p-1}\in \mathbb Q[X]$$ and hence $P(X)$ is the characteristic polynomial of $f$, as well as its minimal polynomial since $P$ is irreducible in $\mathbb Q[X]$. Then since $P$ is irreducible, it is the only invariant factor of $f$ and the structure theory of modules of a principal ideal domain then gives that, as a $\mathbb Q[X]$-module, one has $\mathbb Q^{p-1} \simeq \mathbb Q[X]/(P)$ and $\mathbb Q^{p-1}$ is a simple $\mathbb Q[X]$-module. 
In particular any non zero  $u \in \mathbb Q^{p-1}$ generates $\mathbb Q^{p-1}$ as a $\mathbb Q[X]$-module. Hence since the $\mathbb Q$-subspace generated by  $u, f(u), \ldots , f^{p-2}(u)$ is also a $\mathbb Q[X]$-submodule, we have that these elements generate $\mathbb Q^{p-1}$ and hence also are linearly independent.
The proof for the second matrix is the same as soon as we know that the minimal polynomial of $f$ is $P$, which is easily seen, using that $f^p=1$ and that $1$ is not an eigenvalue of $f$, so that the minimal polynomial of $f$ divides the irreducible polynomial $P$.
\end{proof}

\begin{proof}[Proof of Theorem \ref{infiquot}]
 Let $\pi : k[\Gamma_{M,N}]\rtimes k^{\mathbb Z_M} \rightarrow A $ be surjective Hopf algebra map, with $A$ infinite-dimensional. Then, by Theorem \ref{theo:quot}, $\pi$ induces an isomorphism 
$$k[\Gamma_{M,N}/V] \rtimes_\Phi k^{G}\simeq A$$ 
for $(G, V,\Phi) \in {\rm QD}(\mathbb Z_M \curvearrowright\Gamma_{M,N})$. Since $M$ is prime, either $G$ is trivial or $G=\mathbb Z_M$, and hence $G=\mathbb Z_M$ since $A$ is assumed to be non-cocommutative. We get 
$$k[\Gamma_{M,N}/V] \rtimes_\Phi k^{\mathbb Z_M}\simeq A$$ 
 Then
  Lemma \ref{lem:simple} gives $V \subset \mathbb Z^{(M-1)(N-1)}$, since $N$ is prime  and $A$ is infinite-dimensional. Moreover $V$ is $\mathbb Z_N$-stable (since normal) and $\mathbb Z_M$-stable. The $\mathbb Z_N$ and $\mathbb Z_M$ actions are, in additive notation, implemented by the matrices of Lemma \ref{linearalg}, and hence it follows that if $V \not=0$, then $V$ contains a free abelian subgroup of rank $N-1$ and a free abelian subgroup of rank $M-1$. The quotient of finite rank free abelian group by a subgroup of the same rank is finite, hence if $M=2$ or $N=2$, we have that if $V \not =0$, then $A$ is finite-dimensional, a contradiction. Hence $V=0$ and we are done. 
\end{proof}

\subsection{A quotient datum that is not of the type of Example \ref{ex:char}} We assume that ${\rm char}(k)\not=3$, we put $M=3=N$ and consider the crossed coproduct of the previous subsection
$$k[\Gamma_{3,3}] \rtimes k^{\mathbb Z_3} \simeq k[\mathbb Z^4 \rtimes \mathbb Z_3] \rtimes k^{\mathbb Z_3}$$
We retain the previous notation (see Lemma \ref{lem:gammaMN}):
\begin{itemize}
 \item $\mathbb Z^4$ is seen as the free multiplicative abelian group on $4$ variables $a_{11}$, $a_{12}$, $a_{21}$, $a_{22}$.
\item The first $\mathbb Z_3= \langle t \rangle$-action is given by
$$ t. a_{11}=a_{12}, \ t.a_{12}=a_{11}^{-1}a_{12}^{-1},  \ t.a_{21}=a_{22}, \ t.a_{22}=a_{21}^{-1}a_{22}^{-1}$$
\item  The second $\mathbb Z_3= \langle h \rangle$-action is given by
$$ h\cdot a_{11}=a_{11}^{-1}a_{21}, \ h \cdot a_{12}=a_{12}^{-1}a_{22},  \ h \cdot a_{21}=a_{11}^{-1}, \ h \cdot a_{22}=a_{12}^{-1}, \ h\cdot t= t a_{11}^{-1}a_{12}^{-1}$$
\end{itemize}
For $m \geq 2$, let $N_m= \langle a_{11}^m, a_{12}^m, a_{21}^m, a_{22}^m\rangle \subset \mathbb Z^4$. The group $N_m$ is free abelian of rank $4$, hence for $\alpha, \beta\in k^\times$, there exists a unique group morphism 
\begin{align*}
 \Phi : N_m & \longrightarrow (k^{\mathbb Z_3})^{\times} \\
a_{11}^m, a_{12}^m & \longmapsto \delta_1 + \alpha\delta_{h} + \alpha\beta \delta_{h^2} \\
a_{21}^m, a_{22}^m & \longmapsto \delta_1 + \beta^{-1}\delta_{h} + \alpha \delta_{h^2}
\end{align*}
It is a tedious but straightforward verification to check that for $\alpha^3=1 =\beta^3$, then $(\mathbb Z_3, N_m,\Phi) \in {\rm QD}(\mathbb Z_3 \curvearrowright \mathbb Z^4 \rtimes \mathbb Z_3)$ (in fact any $\Phi$ such that  $(\mathbb Z_3, N_m,\Phi) \in {\rm QD}(\mathbb Z_3 \curvearrowright \mathbb Z^4 \rtimes \mathbb Z_3)$ has the above form). However $\Phi$ has values into $\widehat{\mathbb Z_3}$ only when $\alpha=\beta$. This therefore furnishes the announced example.

\section{Hopf image of a smash coproduct}\label{sec:hopfim}

In this section we show how to describe the Hopf image of a representation of a smash coproduct as above.

\subsection{Hopf images} We begin by recalling the basic facts on Hopf images \cite{bb1}.

Let $A$ be Hopf algebra, let $R$ be an algebra and let $\rho : A \rightarrow R$ be an algebra map.

A factorization of $\rho$
is a triple $(L, q, \varphi)$  where $L$ is a Hopf algebra,
$q : A \rightarrow L$ is a surjective Hopf algebra map
and $\varphi : L \rightarrow A$ is an algebra map, with 
the decomposition $\rho = \varphi  q$.
The category of factorizations
of $\rho$ is defined in the obvious manner and the Hopf image of $\rho$ is defined to be the
final object in this category (hence we can also say that this is a minimal factorization), which is easily shown to exist (see \cite{bb1}).
 
In other words, the Hopf image of $\rho$ is a factorization $(A_\rho , p , \tilde{\rho})$ having the following property:
if $(L,q, \varphi)$ is another 
factorization of $\rho$, there exists a unique
Hopf algebra map $f : L \longrightarrow A_\rho$ such that
$f q = p$ and $\tilde{\rho} f = \varphi$.  
$$\xymatrix{
A \ar[rr]^{\rho} \ar[dr]_{p} \ar@/_/[ddr]_q & & R  \\
& A_\rho \ar[ur]_{\tilde{\rho}} & \\
& L \ar@{-->}[u]_f \ar@/_/[uur]_{\varphi}& 
}$$
The algebra map $\rho : A \rightarrow R$ is said to be inner faithful if $(A, {\rm id}_A, \rho)$ is the Hopf image of $\rho$: this is equivalent to saying that ${\rm Ker}(\rho)$ does not contain any non-zero Hopf ideal, see \cite{bb1}.

Computing a Hopf image is in general a difficult problem. The following cases are well understood, at least from the theoretical viewpoint.

\begin{enumerate}
 \item If $A =k[\Gamma]$ is a group algebra, then the Hopf image of $\rho$ is $k[\Gamma/N]$ where $N={\rm Ker}(\rho_{|\Gamma})$, and the representation is inner faithful if and only if $N=\{1\}$.
\item If $A=k^H$, with $H$ a finite group, $R=k^n$ and the algebra map $\rho$ is given by
 \begin{align*}
  k^H &\longrightarrow k^n \\
  f &\longmapsto (f(h_1), \ldots , f(h_n))
 \end{align*}
for $h_1, \ldots ,h_n \in H$, then the Hopf image of $\rho$
is  $k^{\langle h_1, \ldots ,h_n\rangle}$ and $\rho$ is inner faithful if and only if $H = \langle h_1, \ldots ,h_n\rangle$, see \cite{bb1}. Note that by the semisimplicity and commutativity of $k^H$, this example enables one to describe the Hopf image for any representation $k^H \rightarrow M_n(k)$.
\end{enumerate}

\subsection{Hopf images and smash coproducts} As before, let $H\curvearrowright\Gamma$ be a finite group $H$ acting by automorphisms on a discrete group $\Gamma$, and let $R$ be an algebra. Our aim is to describe the Hopf image of an algebra map $\rho : k[\Gamma]\rtimes k^H \rightarrow R$, therefore unifying the descriptions given at the end of the previous subsection. In fact, to simplify the set-up, we will always assume that $\rho_{|k^H}$ is inner faithful (otherwise, we can factorize $\rho$ by an algebra map $\rho':k[\Gamma] \rtimes k^{H'} \rightarrow R$ with $H'$ a subgroup of $H$ and $\rho'_{|k^{H'}}$ inner faithful, thanks to the last item in the previous subsection). If 
$(H,N,\Phi) \in {\rm QD}(H\curvearrowright\Gamma)$, then we simply denote $(N,\Phi)$ the corresponding element of ${\rm QD}(H\curvearrowright\Gamma)$.

\begin{proposition}\label{hopfim}
  Let $H\curvearrowright\Gamma$ as above and let $\rho :   k[\Gamma] \rtimes k^H\rightarrow R$
be an algebra map such that $\rho_{|k^H}$ is inner faithful. Let 
$$\mathcal E(\rho)=\{(H, N,\Phi)=(N,\Phi) \in {\rm QD}(H\curvearrowright\Gamma) \ | \ 
\forall r \in N, \ \rho( r\#1) = \rho(1\#\Phi(r))\}$$
For any $(N,\Phi) \in \mathcal E(\rho)$, there exists a factorization
$$\xymatrix{k[\Gamma]\rtimes k^H\ar[rr]^{\rho}\ar[rd]_q&& R\\&  k[\Gamma/N] \rtimes_\Phi k^H\ar[ur]_{\tilde{\rho}}&}$$
where if $j : \Gamma/N \rightarrow \Gamma$ is a section of the canonical projection $u: \Gamma \rightarrow \Gamma/N$ with $ju(1)=1$, $q(r\#\delta_h)= u(r)\#\delta_h\Phi(rju(r)^{-1})$ and $\tilde{\rho}(u(r)\#\delta_h)=\rho(ju(r)\#\delta_h)$.

Endow $\mathcal E(\rho)$ with the partial order defined by $(N,\Phi)\leq (M,\Psi) \iff N\subset M$ and $\Psi_{|N}=\Phi$. Then $\mathcal E(\rho)$  admits a maximal element. For any maximal element $(N,\Phi)\in \mathcal E(\rho)$, 
the above factorization is universal and $k[\Gamma/N] \rtimes_\Phi k^H$ is isomorphic to the Hopf image of $ \rho$.
\end{proposition}

\begin{proof}
Let $(N,\Phi) \in \mathcal E(\rho)$.
 The Hopf algebra map $q$ is defined in Proposition \ref{construction}. We have
\begin{align*}\tilde{\rho}q (r\#\delta_h)&=\rho(ju(r)\#\delta_h\Phi(rju(r)^{-1}))=
\rho(1\#\Phi(rju(r)^{-1})) \rho(ju(r)\#\delta_h) \\
&=\rho(rju(r)^{-1}\#1) \rho(ju(r)\#\delta_h) = \rho(r\#\delta_h)
\end{align*}
Hence $\tilde{\rho}q=\rho$ and $\tilde{\rho}$ is an algebra map, and we have our factorization. It immediate that
$\mathcal E(\rho)$ is non empty, that $\leq$ defined above is indeed a partial order on $\mathcal E(\rho)$, and it is an easy verification to check that $ \mathcal E(\rho)$, endowed with this partial order, is inductively ordered. By Zorn's Lemma we can pick a  maximal element $(N,\Phi)$  in 
 $\mathcal E(\rho)$. Let us show that the previous factorization realizes the Hopf image of $\rho$.
So let $(L,p,\overline{\rho})$ be the universal factorization of $\rho$: the universal property of the Hopf image yields a Hopf algebra map $\pi : k[\Gamma/N] \rtimes_\Phi k^H\rightarrow L$ such that the following diagram and all its subdiagrams commute.
$$\xymatrix{
  k[\Gamma] \rtimes k^H \ar[rr]^{\rho} \ar[dr]_{p} \ar@/_/[ddr]_q & &R)  \\
& L \ar[ur]_{\overline{\rho}} & \\
&  k[\Gamma/N] \rtimes_\Phi k^H \ar@{-->}[u]_\pi \ar@/_/[uur]_{\tilde{\rho}}& 
}$$
By construction $\pi$ is surjective, and $\pi_{|k^H}$ is injective since $p$ is (by the inner faithfulness of  $\rho_{|k^H}$). Let 
$$M= \{r \in \Gamma \ | \ \exists f \in k^H \ {\rm with} \ \pi(u(r)\#1) = \pi(1\#f)\}$$ and
$$M'= \{r \in \Gamma \ | \ \exists f \in k^H \ {\rm with} \ p(r\#1) = p(1\#f)\}$$
For $r \in M$, we have 
$$p(r\#1)= \pi q(r\#1)=\pi(u(r)\#\Phi(rju(r)^{-1}))=\pi(1\#\Phi(rju(r)^{-1})f)=p(1\#\Phi(1\#rju(r)^{-1})f)$$ for some $f \in k^H$, hence $r\in M'$. For $r\in M'$, we have
$$\pi(u(r)\#1)=\pi q(r\#\Phi(ju(r)r^{-1}))=p(1\#\Phi(ju(r)r^{-1})f)= \pi(1\#\Phi(ju(r)r^{-1})f)$$
for some $f \in k^H$, and $r \in M$. Hence $M=M'$. We know, by Proposition \ref{quotients}, that $M$ is an $H$-stable normal subgroup of $\Gamma$ and that there exists $\Psi : M \rightarrow C(H)^\times$ such that 
$(M,\Psi) \in  {\rm QD}(H\curvearrowright\Gamma)$ and $p(r\#1)=p(1\#\Psi(r))$ for $r \in M$. For $r \in M$, we have
$$\rho(r\#1)=\overline{\rho}p(r\#1)=\overline{\rho}p(1\#\Psi(r))=\rho(1\#\Psi(r))$$
and hence $(M,\Psi)\in \mathcal E(\rho)$. It is clear from the first description of $M$ that $N \subset M$.
For $r \in N$, we have 
$$p(r\#1) = \pi q(r\#1)=p(1\#\Phi(r))=p(1\#\Psi(r))$$
hence $(N,\Phi)\leq (M,\Psi)$, and we have $N=M$ by maximality of $(N,\Phi)$. It then follows from Lemma \ref{iso}
that $\pi$ is injective, and hence is an isomorphism.
\end{proof}

\begin{remark}{\rm 
 It is in fact possible to avoid the use of Zorn's Lemma in the previous proof, using the existence of the Hopf image. We found the use of Zorn's Lemma more convenient to formulate the proof. A drawback is that the description is not very explicit (but this would not be more explicit without Zorn's Lemma).  } 
\end{remark}

 We now present two situations where the Hopf image has a more explicit description. 

\begin{corollary}\label{corohopfim}
 Let $H\curvearrowright\Gamma$ as above and let $\rho :  k[\Gamma] \rtimes k^H  \rightarrow R$
be a representation such that $\rho_{|k^H}$ is inner faithful. Consider the $H$-stable normal subgroup of $\Gamma$
$$N= \{ r \in \Gamma \ | \ \forall h \in H, \ \exists f \in (k^H)^\times \ {\rm with} \ \rho(h\cdot r\#1)=\rho(1\#f)\}$$
 and assume that there exists $\Phi : N \rightarrow (k^H)^\times$ such that $(N,\Phi) \in \mathcal E(\rho)$. Then the Hopf image of $\rho$ is isomorphic with $k[\Gamma/N] \rtimes_\Phi k^H$.
\end{corollary}

\begin{proof}
 For $(M,\Psi) \in  \mathcal E(\rho)$, we have $M \subset N$. Hence if $(N,\Phi)\leq (M,\Psi)$, then $N=M$ and $\Phi=\Psi$. This shows that $(N,\Phi)$ is maximal, and the previous result finishes the proof. 
\end{proof}

\begin{corollary}\label{corohopfim2}
 Let $H\curvearrowright\Gamma$ as above and let $\rho :  k[\Gamma] \rtimes k^H  \rightarrow R$
be a representation such that $\rho_{|k^H}$ is faithful. Let
$$N_0= \{ r \in \Gamma \ | \ \exists f \in k^H \ {\rm with} \ \rho(r\#1)=\rho(1\#f)\}$$
This is a normal subgroup of $\Gamma$, and the faithfulness assumption on $\rho_{|k^H}$ yields a group morphism $\Phi : N_0 \rightarrow (k^H)^\times$ such that $\rho(r\#1)=\rho(1\#\Phi(r))$ for any $r \in N_0$.
Now put
$$N= \{r \in N_0 \ | \ \forall h,k,l \in H, \ h.r \in N_0 \ {\rm and} \ 
\Phi(k.r)(lh)=\Phi((l^{-1}h).r)(h)\Phi(k.r)(l)\}$$
Then $N$ a normal and $H$-stable subgroup of $\Gamma$, $(N,\Phi) \in \mathcal E(\rho)$ and the Hopf image of $\rho$ is isomorphic with $k[\Gamma/N] \rtimes_\Phi k^H$.
\end{corollary}

\begin{proof}
It is a direct verification to check that $N$ is a normal and $H$-stable subgroup of $\Gamma$, that 
$(N,\Phi) \in {\rm QD}(H\curvearrowright\Gamma)$ and hence that $(N,\Phi) \in \mathcal E(\rho)$.
 For $(M,\Psi) \in  \mathcal E(\rho)$, we have $M \subset N$. Hence if $(N,\Phi)\leq (M,\Psi)$, then $N=M$ and $\Phi=\Psi$. This shows that $(N,\Phi)$ is maximal, and Proposition \ref{hopfim} finishes the proof. 
\end{proof}

\section{Examples}\label{sec:exhopim}

We  illustrate the results of the previous section using the examples of Section \ref{sec:example}. 
We assume that $k$ has characteristic zero here.

\subsection{Construction of the representations}
Let $M,N \geq 2$. As in \cite{bb3}, we fix a matrix $Q =(Q_{ic})=\in M_{MN}(k^\times)$ with
$Q_{0c}=1=Q_{i0}$ for any $i$, $c$ (the indices are taken modulo $M$, $N$, respectively). To $Q$ we associate the matrix $\theta=(\theta_{ic}) \in M_{MN}(k^\times)$ defined by 
$\theta_{ic}=\frac{Q_{i-1,c}Q_{i,c-1}}{Q_{ic}Q_{i-1,c-1}}$. We have
$$\prod_{c=0}^{N-1}\theta_{ic}=1=\prod_{j=0}^{M-1}\theta_{jd}, \ 0 \leq i \leq M-1, \ 0 \leq d \leq N-1$$
We denote by $\epsilon_0, \ldots , \epsilon_{N-1}$ the canonical basis of $k^{N}$. 
We consider the Hopf algebra $k[\Gamma_{M,N}]\rtimes k^{\mathbb Z_M}$ of Subsection \ref{subsec:main} and 
we will be interested in the representation 
$$\rho_Q : k[\Gamma_{M,N}]\rtimes k^{\mathbb Z_M} \longrightarrow {\rm End}(k^N)$$
defined as follows: for $0 \leq i \leq M-1$, we have 
$$\rho_{Q}(g_i\#1)(\epsilon_c) = \theta_{ic} \epsilon_{c-1}$$
and for $f \in k^{\mathbb Z_M}$, we have 
$$\rho_Q(1\#f) =f(h){\rm id}, \ {\rm where} \ \mathbb Z_M=\langle h \rangle$$ 
The representation $\rho_Q$ is a constituent of the representation $\pi_Q$ in \cite{bb3}, to which we will restrict here
(note however that inner faithfulness of $\rho_Q$ implies inner faithfulness of $\pi_Q$).

It is clear that $\rho_{Q_{|k^{\mathbb Z_M}}}$ is inner faithful, so we can use the statements of the previous section.

Recall \cite{bb3} that we say that $p_1,\ldots ,p_m\in k^\times$ are root independent if for any $r_1,\ldots, r_m\in\mathbb Z$:
$$p_1^{r_1}\ldots p_m^{r_m}=1\implies r_1=\ldots=r_m=0$$
It is shown in \cite{bb3} that if the elements $Q_{ic}$, $1\leq i\leq M-1$, $1 \leq c \leq N-1$ are root independent, then the representation $\rho_Q$ is inner faithful.
Our main aim is to show that, at least in some situations, the root independence assumption can be weakened, as follows.

\begin{theorem}\label{nonroot}
 Assume that $M=2$ and $N$ is prime, or that $M$ is prime and $N=2$. If one the elements $Q_{ic}$ is not a root of unity, then the representation $\rho_Q$ is inner faithful.
\end{theorem}


\subsection{Preliminaries and notation} We now develop some preliminary material. We retain the previous notation.
For $R=(R_{ic})$, $1\leq i \leq M-1$, $1 \leq c \leq N-1$, $R_{ic} \in \mathbb Z$, put $S_{jc}= R_{jc} + \sum_{i=1}^{M-1}R_{ic}$,
$$\alpha(R,0)=  \left(\prod_{j=1}^{M-1}\prod_{c=1}^{N-1}\theta_{j,c}^{S_{j,c}}\right)$$
 and for $1 \leq d \leq N-1$,
$$\alpha(R, d)= \left(\prod_{j=1}^{M-1}\theta_{j,-d}^{-S_{j,-d}}\right) \left(\prod_{j=1}^{M-1}\prod_{c=1, c\not= -d}^{N-1}\theta_{j,c}^{S_{j,c-d}-S_{j,-d}}\right) $$
The following result is a direct verification.

\begin{lemma}
 For any $0\leq  d \leq N-1$, the map 
\begin{align*}
 \alpha(- ,d) :  \mathbb Z^{(M-1)(N-1)}&\longrightarrow k^\times \\
R & \longmapsto \alpha(R,d)
\end{align*}
is a group morphism.
\end{lemma}

There is moreover an action of $\mathbb Z_M=\langle h \rangle$ on $\mathbb Z^{(M-1)(N-1)}$ given on the standard basis $\epsilon_{ic}$, $1 \leq i \leq M-1$, $1 \leq c \leq N-1$,  by $h\cdot \epsilon_{ic}=\epsilon_{i,c+1}-\epsilon_{1,c}$ (the indices are taken modulo $M$, $N$). 
This is in fact the same action as the one in Lemma \ref{lem:gammaMN}, but written additively. For $0 \leq l \leq M-1$ and $R=(R_{ic})\in \mathbb Z^{(M-1)(N-1)} $, we note $l\cdot R = h^l\cdot R$.

\begin{definition}
For $0 \leq l \leq M-1$, the groups $E_Q^l\subset \mathbb Z^{(M-1)(N-1)}$ and $I_Q^l\subset (k^\times)^{N-1} $ are the respective kernel and image of the group morphism
\begin{align*}
 \mathbb Z^{(M-1)(N-1)}&\longrightarrow (k^\times)^{N-1} \\
R & \longmapsto \left(\alpha(l\cdot R,0)\alpha(l\cdot R,d)^{-1}\right)_{1 \leq d \leq N-1}
\end{align*}
and we put $E_Q = \cap_{l=0}^{M-1}E_Q^l$
\end{definition}

\begin{lemma}\label{lem:roots}
\begin{enumerate}
 \item If the elements $Q_{ic}$, $1\leq i\leq M-1$, $1 \leq c \leq N-1$ are root independent, then $E_Q^0=(0)=E_Q$.
\item If one of the elements $Q_{ic}$ is not a root of unity, then the group $I_Q^0$ is infinite.
\end{enumerate}
\end{lemma}

\begin{proof} (1) One checks first that if the elements $Q_{ic}$, $1\leq i\leq M-1$, $1 \leq c \leq N-1$ are root independent, then so are the elements $\theta_{ic}$, $1\leq i\leq M-1$, $1 \leq c \leq N-1$, and then the verification that $E_Q^0=(0)$ is immediate using the root independence of those elements.
 
(2) Using the standard basis of the free abelian group 
$\mathbb Z^{(M-1)(N-1)}$, we see that $I_Q^0$ is the subgroup of $(k^\times)^{N-1}$ generated by the elements
$$(\theta_{ic}\theta_{0c}^{-1}\theta_{i,c+d}^{-1}\theta_{0,c+d})_{1\leq d \leq N-1}, \ 1 \leq i \leq M-1, \ 1 \leq c \leq N-1$$
Denote by $\mu_{\infty}$ the group of roots of unity in $k^\times$ and assume that $I_Q^0$ is finite.
Then for any $1 \leq i\leq M-1$ and $1 \leq c, d \leq N-1$ we have
$$\theta_{ic}\theta_{0c}^{-1}\theta_{i,c+d}^{-1}\theta_{0,c+d} \in \mu_\infty$$
and in particular for any $1 \leq c, d \leq N-1$, we have
$$\prod_{i=1}^{N-1}\theta_{ic}\theta_{0c}^{-1}\theta_{i,c+d}^{-1}\theta_{0,c+d}= 
\theta_{0,c}^{-N} \theta_{0,c+d}^N \in \mu_\infty \Rightarrow \theta_{0,c+d}\theta_{0,c}^{-1}\in \mu_\infty$$
Then we have for any $1 \leq c \leq N-1$
$$\prod_{d=1}^{N-1} \theta_{0,c+d}\theta_{0,c}^{-1}= \theta_{0,c}^{-N} \in \mu_{\infty}\Rightarrow \theta_{0,c} \in \mu_{\infty}$$
From this we deduce easily that $\theta_{ic} \in \mu_\infty$ for any $i,c$, and then that $Q_{ic} \in \mu_\infty$ for any $i,c$ as well.
\end{proof}

\subsection{}
We come back to the study of the representation $\rho_Q$.
According to Proposition \ref{hopfim} and Corollary \ref{corohopfim}, we need to study the group 
\begin{align*}N_Q &= \{ r \in \Gamma_{M,N} \ | \ \forall y \in \mathbb Z_M, \ \exists f \in (k^{\mathbb Z_M})^{\times} \ {\rm with} \ \rho_Q(y\cdot r \# 1)=f(h)1\} \\
   &= \{ r \in \Gamma_{M,N} \ | \ \forall y \in \mathbb Z_M, \ \exists \lambda \in k^* \ {\rm with} \ \rho_Q(y\cdot r \# 1)=\lambda 1\}
\end{align*}

\begin{lemma}\label{lem:hopfimgamma} 
 The subgroup $N_Q$ is the subgroup of $T=\langle a_{ic}, \ 1\leq i \leq M-1, \ 1 \leq c \leq N-1 \rangle$ formed by elements
$$a= \prod_{i=1}^{M-1}\prod_{c=1}^{N-1}a_{ic}^{R_{ic}}$$
for which we have $R=(R_{ic}) \in E_Q$. Moreover the Hopf image of $\rho_Q$ is isomorphic to  
$$k[T/U\rtimes \mathbb Z_N] \rtimes_\Phi k^{\mathbb Z_M}$$ for some quotient datum $(U,\Phi)$, where $U \subset N_Q$.
\end{lemma}

\begin{proof}
One sees easily that an element in $N_Q$ belongs to $T$, and we have $\rho_Q(a_{ic}\#1)(\epsilon_d) = \theta_{i,c+d}\theta_{0,c+d}^{-1}\epsilon_d$. From this we see that for $a$ as above, we have
$\rho_Q(a\#1)(\epsilon_d) = \alpha(R,d) \epsilon_d$, and hence $N_Q$ is indeed the announced subgroup. By Proposition \ref{hopfim}, the Hopf image of $\rho_Q$ is isomorphic to 
$k[\Gamma_{M,N}/U] \rtimes_\Phi k^{\mathbb Z_M}$ for some subgroup $U \subset N_Q$, with $\Gamma_{M,N}/U = T/U\rtimes \mathbb Z_N$ by the first assertion.
\end{proof}

From this, choosing $Q$ such that $E_Q=(0)$, we see that $T$ is indeed free abelian on the elements $a_{ic}$, $1\leq  i \leq M-1$, $1 \leq c \leq N-1$ (Lemma \ref{lem:gammaMN}). In general we also see that the groups $E_Q$ and $N_Q$ are isomorphic, and that $\mathbb Z^{(M-1)(N-1)}/E_Q\simeq T/N_Q$. 

From this, we first recover Theorem 4.6 from \cite{bb3} in the case of cyclic groups.

\begin{corollary}\label{cor:vanisheq}
 If $E_Q=(0)$, then the representation $\rho_Q$ is inner faithful.
\end{corollary}

\begin{proof}
 If $E_Q$ is trivial, so is $N_Q$, and the result follows from Lemma \ref{lem:hopfimgamma}.
\end{proof}

\begin{corollary}
 If $I_Q^0$ is infinite, then the Hopf image of the representation $\rho_Q$ is infinite-dimensional.
\end{corollary}

\begin{proof}
 Again the Hopf image of $\rho_Q$ is isomorphic to 
$k[T/U\rtimes \mathbb Z_N] \rtimes_\Phi k^{\mathbb Z_M}$ for some subgroup $U \subset N_Q$. We have $I_Q^0 \simeq \mathbb Z^{(M-1)(N-1)}/E_Q^0$, so $[\mathbb Z^{(M-1)(N-1)}:E_Q]=[T:N_Q]$ is infinite, as well as $[T:U]$, and we are done.
\end{proof}

We can also prove Theorem \ref{nonroot} now.

\begin{proof}[Proof of Theorem \ref{nonroot}]
The group $I_Q^0$ is infinite by Lemma \ref{lem:roots}, hence by the previous corollary the Hopf image of $\rho_Q$, isomorphic to $k[\Gamma_{M,N}/U] \rtimes_\Phi k^{\mathbb Z_M}$,  is infinite-dimensional. By Theorem \ref{infiquot}, either $U$ is trivial, and we are done, either $k[\Gamma_{M,N}/U] \rtimes_\Phi k^{\mathbb Z_M}$ is cocommutative. In this case the $\mathbb Z_M$-action on $\Gamma_{M,N}/U$ is trivial, and since it permutes cyclically the generators, the quotient group $\Gamma_{M,N}/U$ is finite cyclic, and $k[\Gamma_{M,N}/U] \rtimes_\Phi k^{\mathbb Z_M}$ is finite-dimensional, a contradiction.
\end{proof}

\subsection{Example at small indices} We end the paper with some precise results at small indices $M$ and $N$.
We begin by the case $M=2=N$. 
We have then 
$$Q = \begin{pmatrix} 1 & 1 \\ 1 & q\end{pmatrix} \ \ {\rm and} \ \ \theta = \begin{pmatrix} q^{-1} & q \\ q & q^{-1}\end{pmatrix}$$
       for some $q \in k^*$, and we simply denote $\rho_Q$ by $\rho_q$. 
We retain the notation of the beginning of Section \ref{sec:example}.

\begin{proposition}
Let $A_q$ denote the Hopf image of $\rho_q : k[\Gamma_{2,2}]\rtimes k^{\mathbb Z_2} \rightarrow {\rm End}(k^2)\simeq M_2(k)$, and let $m=o(q)$.
\begin{enumerate}
 \item If $m=\infty$, then $A_q\simeq  k[\Gamma_{2,2}]\rtimes k^{\mathbb Z_2}$.
\item If $m \not  \in 2 \mathbb N$, then $A_q \simeq A(m)$. 
\item If  $m \in 2 \mathbb N$ and $m \not  \in 4 \mathbb N$, then $A_q \simeq A(\frac{m}{2})$. 
\item If $m \in 4\mathbb N$, then $A_q \simeq B(\frac{m}{4})$.
\end{enumerate}
In particular, we have $\dim(A_q)=4o(q^4)$. 
\end{proposition}

\begin{proof}
 (1) follows from Corollary \ref{cor:vanisheq}. We assume now that $q$ is a root of unity. We have, in matrix form
$$\rho_q(a_{11}\#1) = \begin{pmatrix} q^{-2} & 0 \\ 0 & q^{2}\end{pmatrix}$$
Thus the subgroup $N_Q$ in Lemma \ref{lem:hopfimgamma} is formed by the elements $\{a_{11}^k, \  k \in \mathbb Z, \ m|4k\}$. 

(2) Assume that $m \not \in 2\mathbb N$. Then $N_Q=\langle a_{11}^{m}\rangle$. For $\Phi$ the trivial map, we easily see that $(N_Q,\Phi) \in \mathcal E(\rho_q)$, and hence we have $A_q \simeq A(m)$ by Corollary \ref{corohopfim}.

(3) Assume that $m \in 2 \mathbb N$ and $m \not  \in 4 \mathbb N$. Then $N_Q=\langle a_{11}^{\frac{m}{2}}\rangle$. For $\Phi$ the trivial map, we  see that $(N_Q,\Phi) \in \mathcal E(\rho_q)$, and hence we have $A_q \simeq A(\frac{m}{2})$ by Corollary \ref{corohopfim}.

(4) Assume that $m \in 4\mathbb N$. Then $N_Q=\langle a_{11}^{\frac{m}{4}}\rangle$. Consider, as in the beginning of Section \ref{sec:example}, $\Phi_{\frac{m}{4}} : \langle (a_{11})^{\frac{m}{4}} \rangle\simeq \mathbb Z \rightarrow \widehat{\mathbb Z_2} =\langle \chi \rangle$, the unique group morphism with $\Phi_m(a_{11}^{\frac{m}{4}})=\chi$ (recall that $g_0g_1=a_{11}$). It is immediate to check that $(N_Q,\Phi_{\frac{m}{4}}) \in \mathcal E(\rho_q)$, and hence we have $A_q \simeq B(\frac{m}{4})$ by Corollary \ref{corohopfim}.

The last assertion is immediate.
\end{proof}

As a last example, we consider the case $M=3$, $N=2$. We then have  
$$Q = \begin{pmatrix} 1 & 1 \\ 1 & p \\ 1 & q\end{pmatrix} \ \ {\rm and} \ \
\theta = \begin{pmatrix} q^{-1} & q \\ p & p^{-1} \\ qp^{-1} & pq^{-1}\end{pmatrix}$$
If $p$ or $q$ is not a root of unity, we know from Theorem \ref{nonroot} that $\rho_Q$ is inner faithful. In the root of unity case, we have  the following particular result.

\begin{proposition}
 Assume that $p$ and $q$ are roots of unity, and let $m=o(p^2)$ and $n=o(q^2)$. Denote by $A_Q$ the Hopf image of $\rho_Q$. 
\begin{enumerate}
\item 
If 
${\rm GCD}(m,n)=1={\rm GCD}(m,3)={\rm GCD}(n,3)$, then $A_Q$ is isomorphic to a smash coproduct
$$k[(\mathbb Z_{mn}\times \mathbb Z_{mn})\rtimes \mathbb Z_2] \rtimes k^{\mathbb Z_3}$$ 
\item If $p^2=q^2$ and ${\rm GCD}(m,3)=1$, then $A_Q$ is isomorphic to a smash coproduct
$$k[(\mathbb Z_{m}\times \mathbb Z_{m})\rtimes \mathbb Z_2] \rtimes k^{\mathbb Z_3}$$ 
\item If $p^2=q^2$ and $3|m$, then $A_Q$ is isomorphic to a smash coproduct
$$k[(\mathbb Z_{m}\times \mathbb Z_{\frac{m}{3}})\rtimes \mathbb Z_2] \rtimes k^{\mathbb Z_3}$$ 
\end{enumerate}
\end{proposition}

\begin{proof}
 In matrix form, we have 
$$\rho_q(a_{11}\#1) = \begin{pmatrix} p^{-2} & 0 \\ 0 & q^{2}\end{pmatrix}, \ 
\rho_q(a_{21}\#1) = \begin{pmatrix} p^{2} & 0 \\ 0 & q^{4}\end{pmatrix}$$
Hence the group $N_Q$ consists of elements $a_{11}^\alpha a_{21}^{\beta}$ for which we have
$$(p^2)^{-\alpha +\beta}=(q^2)^{\alpha + 2 \beta}, \ (p^2)^{2 \alpha +\beta}=(q^2)^{\alpha -\beta}$$

(1) Our assumptions imply that $N_Q$ consists of elements $a_{11}^\alpha a_{21}^{\beta}$ with $\alpha, \beta \in mn \mathbb Z$. Taking $\Phi : N_Q \rightarrow k^{\mathbb Z_3}$ the trivial map, we see that $(N_Q,\Phi) \in \mathcal E(\rho)$, and we conclude by Corollary \ref{corohopfim}.

(2) Our assumptions imply that $N_Q$ consists of elements $a_{11}^\alpha a_{21}^{\beta}$ with $\alpha, \beta \in m \mathbb Z$, and we conclude as in the previous case.

(3) Here our assumption  imply that $N_Q$ consists of elements $a_{11}^\alpha a_{21}^{\beta}$ with $$\alpha, \beta \in \{(-2k\frac{m}{3}+ml, k\frac{m}{3}), \ k,l \in \mathbb Z\}=\mathbb Z(m,0) +\mathbb Z(-2\frac{m}{3},\frac{m}{3})=E_Q\subset \mathbb Z^2$$
We then have $T/N_Q\simeq \mathbb Z^2/E_Q\simeq \mathbb Z_{m}\times \mathbb Z_{\frac{m}{3}}$ (by the standard theory of finitely generated abelian groups), and we conclude as in the previous cases.
\end{proof}

\end{document}